\documentclass[11pt]{amsart}

\usepackage[T1]{fontenc}
\usepackage[utf8]{inputenc}
\usepackage{lmodern}
\usepackage{amsmath,amssymb,amsthm,mathtools}
\usepackage{enumitem}
\usepackage{aliascnt}
\usepackage{microtype}
\usepackage{hyperref}

\hypersetup{
  colorlinks=true,
  linkcolor=blue,
  citecolor=blue,
  urlcolor=blue
}

\theoremstyle{plain}
\newtheorem{theorem}{Theorem}[section]
\newaliascnt{corollary}{theorem}
\newtheorem{corollary}[corollary]{Corollary}
\aliascntresetthe{corollary}
\newaliascnt{proposition}{theorem}
\newtheorem{proposition}[proposition]{Proposition}
\aliascntresetthe{proposition}
\newaliascnt{lemma}{theorem}
\newtheorem{lemma}[lemma]{Lemma}
\aliascntresetthe{lemma}

\theoremstyle{definition}
\newaliascnt{definition}{theorem}
\newtheorem{definition}[definition]{Definition}
\aliascntresetthe{definition}
\newaliascnt{problem}{theorem}
\newtheorem{problem}[problem]{Problem}
\aliascntresetthe{problem}
\newaliascnt{remark}{theorem}
\newtheorem{remark}[remark]{Remark}
\aliascntresetthe{remark}
\newaliascnt{example}{theorem}

\aliascntresetthe{example}

\usepackage[capitalize,noabbrev]{cleveref}
\crefname{problem}{Problem}{Problems}
\Crefname{problem}{Problem}{Problems}

\newcommand{\Aut}{\operatorname{Aut}}

\newcommand{\id}{\operatorname{id}}

\newcommand{\Cay}{\operatorname{Cay}}
\newcommand{\Grig}{\mathfrak G}
\newcommand{\eps}{\varepsilon}

\title[Closed image characterizations]{Closed Image Characterizations of Locally Finite Groups via Cellular Automata}
\author{Jiang Yang}
\address{School of Mathematical Sciences, Guangxi Minzu University, Nanning, China}
\email{yangjiangdy@126.com}
\date{}

\subjclass[2020]{37B15, 20F65, 20F50, 16S34}
\keywords{Cellular automata; linear cellular automata; closed image property; locally finite groups; infinite alphabets; prodiscrete topology; Grigorchuk group}

\begin{document}

\begin{abstract}
We prove that a group $G$ is locally finite if and only if, for some (equivalently, every) infinite set $A$, every cellular automaton $A^G\to A^G$ has closed image in the prodiscrete topology.  Equivalently, this holds if and only if every linear cellular automaton $V^G\to V^G$ has closed image for some pair $(K,V)$ with $V$ infinite-dimensional over the field $K$ (equivalently, for every such pair).  This gives affirmative answers to Open Problems 6 and 7 of Ceccherini-Silberstein and Coornaert.  More precisely, if $G$ is not locally finite, then for every infinite set $A$ there is a finite-memory cellular automaton $A^G\to A^G$ with non-closed image, and for every field $K$ and every infinite-dimensional $K$-vector space $V$ there is such a linear cellular automaton $V^G\to V^G$.  The common obstruction is constructed on a countable direct-sum alphabet from an infinite ray in a locally finite Cayley graph.  A direct-summand argument gives arbitrary vector-space alphabets, while an alphabet-retract principle gives arbitrary infinite set alphabets.
\end{abstract}

\maketitle

\section{Introduction}

Let $G$ be a group and let $A$ be a set.  A cellular automaton over $G$ with alphabet $A$ is a map $A^G\to A^G$ determined by a finite memory set and a local rule.  In the classical one-dimensional setting, this is the content of Hedlund's characterization of endomorphisms of shift systems as sliding block codes \cite{Hedlund1969}; in the group setting, the corresponding formulation says that cellular automata are precisely the $G$-equivariant uniformly continuous self-maps of the full configuration space $A^G$ endowed with its prodiscrete uniform structure, and for finite alphabets this is the Curtis--Hedlund--Lyndon theorem.  When $A$ is a vector space $V$ over a field $K$ and the local rule is $K$-linear, one obtains a linear cellular automaton.  The finite-dimensional linear theory includes a Garden of Eden theorem over amenable groups \cite{CeC2006}; moreover, finite-dimensional linear cellular automata are represented by matrices over group algebras, connecting their dynamics with the algebraic structure of $K[G]$ \cite{CAG2,ExercisesCAG}.

This paper concerns the closed image property.  In the usual ZFC framework, for a finite alphabet $A$, Tychonoff compactness gives that $A^G$ is compact, so the image of every cellular automaton $A^G\to A^G$ is closed.  In the linear setting, if $V$ is finite-dimensional over $K$, then every linear cellular automaton
\[
        \tau\colon V^G\longrightarrow V^G
\]
has closed image in the prodiscrete topology, for every group $G$ \cite[Theorem 8.8.1]{CAG2}.  These finiteness hypotheses cannot simply be removed.  Ceccherini-Silberstein and Coornaert proved that if $G$ is non-periodic and $V$ is infinite-dimensional, then some linear cellular automaton $V^G\to V^G$ has non-closed image \cite[Theorem 1.5]{CeC2011}.  Conversely, if $G$ is locally finite, then every cellular automaton over $G$ has closed image for every alphabet \cite[Proposition 6.1]{CeC2011}.

The remaining regime is therefore the periodic but non-locally finite case.  The linear question already appeared in the first edition of \emph{Cellular Automata and Groups} \cite[Open Problems]{CAG1}; the second edition formulates the two adjacent questions in this regime as OP-6 and OP-7 \cite[Open Problems, OP-6 and OP-7, p.~527]{CAG2}.

\begin{problem}[Ceccherini-Silberstein--Coornaert, OP-6]\label{prob:op6}
Let $G$ be a periodic group which is not locally finite and let $A$ be an infinite set.  Does there exist a cellular automaton
\[
        \tau\colon A^G\longrightarrow A^G
\]
whose image $\tau(A^G)$ is not closed in $A^G$ with respect to the prodiscrete topology?
\end{problem}

\begin{problem}[Ceccherini-Silberstein--Coornaert, OP-7]\label{prob:op7}
Let $G$ be a periodic group which is not locally finite, let $K$ be a field, and let $V$ be an infinite-dimensional $K$-vector space.  Does there exist a linear cellular automaton
\[
        \tau\colon V^G\longrightarrow V^G
\]
whose image $\tau(V^G)$ is not closed in $V^G$ with respect to the prodiscrete topology?
\end{problem}

We answer both questions affirmatively and, in fact, show that periodicity plays no role.  Throughout the paper, all set-theoretic statements are made in the usual ZFC framework; the exact dependence on choice and the corresponding ZF-qualified form are recorded in \cref{rem:set-theoretic-convention}.

\begin{theorem}[Non-locally finite groups]\label{thm:main-intro}
Let $G$ be a non-locally finite group.
\begin{enumerate}[label=\textup{(\alph*)}]
\item For every infinite set $A$, there exists a finite-memory cellular automaton
\[
        \tau\colon A^G\longrightarrow A^G
\]
whose image is not closed in the prodiscrete topology.
\item For every field $K$ and every infinite-dimensional $K$-vector space $V$, there exists a finite-memory linear cellular automaton
\[
        \tau\colon V^G\longrightarrow V^G
\]
whose image is not closed in the prodiscrete topology.
\end{enumerate}
\end{theorem}

\begin{corollary}[Answers to OP-6 and OP-7]\label{cor:open-problems}
Let $G$ be a periodic group which is not locally finite.  Then the answers to \cref{prob:op6,prob:op7} are affirmative.
\end{corollary}

The construction yields a particularly clean characterization of local finiteness.

\begin{theorem}[Closed image characterization]\label{thm:characterization-intro}
For a group $G$, the following conditions are equivalent:
\begin{enumerate}[label=\textup{(\roman*)}]
\item $G$ is locally finite;
\item for every set $A$, every cellular automaton $\tau\colon A^G\to A^G$ has closed image in the prodiscrete topology;
\item for every infinite set $A$, every cellular automaton $\tau\colon A^G\to A^G$ has closed image;
\item for some infinite set $A$, every cellular automaton $\tau\colon A^G\to A^G$ has closed image;
\item for every field $K$, every $K$-vector space $V$, and every linear cellular automaton $\tau\colon V^G\to V^G$, the image $\tau(V^G)$ is closed;
\item for some field $K$ and some infinite-dimensional $K$-vector space $V$, every linear cellular automaton $\tau\colon V^G\to V^G$ has closed image.
\end{enumerate}
All closures are taken in the corresponding prodiscrete topology.
\end{theorem}

We now explain the common mechanism.  Since $G$ is not locally finite, it contains a finitely generated infinite subgroup $H$.  A locally finite Cayley graph of $H$ contains an infinite simple ray
\[
        p_0=1_G,
        \qquad
        p_n=s_1s_2\cdots s_n\quad(n\geq 1),
\]
where the labels $s_n$ lie in a fixed finite generating set $S$.  In the previously known non-periodic case, an element of infinite order supplies such a ray with constant label.  Periodic groups have no such element.  The point of the present construction is to allow the labels $s_n$ to vary, and to encode the $n$-th label into the $n$-th internal coordinate of an infinite-dimensional alphabet.

First take the countable direct sum
\[
        E=\bigoplus_{n\geq 0}K e_n.
\]
Let $R\colon E\to E$ be the right shift $R(e_n)=e_{n+1}$, let $P_0$ be the projection onto $Ke_0$, and, for each $s\in S$, let $P_s$ be the coordinate selector which keeps precisely those basis vectors $e_n$ with $n\geq 1$ and $s_n=s$.  Define
\begin{equation}\label{eq:intro-tau}
        \tau_E(x)(g)
        =P_0(x(g))-R(x(g))+\sum_{s\in S}P_s(x(gs)),
        \qquad x\in E^G,\ g\in G.
\end{equation}
This is a linear cellular automaton with memory set $\{1_G\}\cup S$.  If $y\in E^G$ is the constant configuration $y(g)=e_0$, then every finite restriction of $y$ is realized by some element of $\tau_E(E^G)$, while a genuine global preimage would force
\[
        x(g)=e_0+e_1+e_2+\cdots
\]
for every $g\in G$, which is not an element of the algebraic direct sum $E$.  Hence
\[
        y\in\overline{\tau_E(E^G)}\setminus\tau_E(E^G).
\]

There are then two alphabet-extension steps.  If $V$ is an arbitrary infinite-dimensional $K$-vector space, choose a complement $V=E\oplus W$ and extend $\tau_E$ by the identity on $W^G$.  For an arbitrary infinite set $A$, take $K=\mathbb F_2$, so that the underlying set of $E$ is countably infinite; in ZFC, embed this set into $A$ and retract $A$ onto the embedded copy.  Composing the local rule with the retraction and the embedding transfers the non-closed image obstruction to $A^G$.  This second step is isolated below as an alphabet-retract principle.

The first Grigorchuk group $\Grig$ is finitely generated, infinite, and periodic.  Hence it is not locally finite, and \cref{thm:main-intro} applies to both ordinary and linear cellular automata on $\Grig$.  It is therefore a concrete test case for both OP-6 and OP-7.

\section{Preliminaries}\label{sec:preliminaries}

\subsection{Cellular automata and the prodiscrete topology}

Let $G$ be a group and let $A$ be a set.  The configuration space $A^G$ is the set of all maps $x\colon G\to A$.  It is equipped with the prodiscrete topology, namely the product topology obtained by giving $A$ the discrete topology.  For $y\in A^G$ and a finite subset $\Omega\subset G$, put
\[
        \mathcal N(y,\Omega)
        =\{z\in A^G:z|_\Omega=y|_\Omega\}.
\]
These sets form a neighborhood basis of $y$.  Hence, for a subset $X\subset A^G$, one has $y\in\overline X$ if and only if for every finite $\Omega\subset G$ there exists $x\in X$ such that $x|_\Omega=y|_\Omega$.

The left shift action of $G$ on $A^G$ is given by
\[
        (gx)(h)=x(g^{-1}h),
        \qquad g,h\in G.
\]
A map $\tau\colon A^G\to A^G$ is a cellular automaton if there exist a finite subset $M\subset G$ and a map
\[
        \mu\colon A^M\longrightarrow A
\]
such that
\begin{equation}\label{eq:ca-definition}
        \tau(x)(g)=\mu\bigl((g^{-1}x)|_M\bigr)
\end{equation}
for all $x\in A^G$ and $g\in G$.  Since $(g^{-1}x)(m)=x(gm)$, the value $\tau(x)(g)$ depends only on the restriction of $x$ to the finite set $gM$.  If $A=V$ is a vector space over a field $K$ and $\mu$ is $K$-linear, then $\tau$ is a linear cellular automaton.

For any map $f\colon A\to B$ and any set $I$, write
\[
        f^I\colon A^I\longrightarrow B^I,
        \qquad
        f^I(x)(i)=f(x(i)),
\]
for the induced coordinatewise map.  This notation will be used for $I=G$ and for finite memory sets.

A group $G$ is locally finite if every finitely generated subgroup of $G$ is finite.  Equivalently, $G$ is not locally finite if and only if it contains a finitely generated infinite subgroup.

\begin{remark}[Set-theoretic convention]\label{rem:set-theoretic-convention}
We work in the usual ZFC framework.  Two elementary consequences of this convention are used below: every infinite set contains a countably infinite subset, and every subspace of a vector space has a linear complement.  Thus the alphabet-retract step transfers the countable alphabet obstruction to every infinite alphabet, and the direct-summand step transfers the countable direct-sum construction to every infinite-dimensional vector-space alphabet.  If one works in ZF without choice, the same arguments still give the ordinary-alphabet conclusion for every Dedekind-infinite alphabet, that is, every alphabet containing a countably infinite subset, and they give the linear conclusion for every vector space containing a complemented copy of $\bigoplus_{n\geq 0}K e_n$.  The locally finite closed-image direction and the countable direct-sum obstruction itself do not use these choice consequences.
\end{remark}

\subsection{The locally finite direction}

We include the standard proof that locally finite groups have the closed image property for arbitrary alphabets.

\begin{proposition}[Closed images over locally finite groups]\label{prop:locally-finite-closed}
Let $G$ be locally finite and let $A$ be an arbitrary set.  Then every cellular automaton $\tau\colon A^G\to A^G$ has closed image in the prodiscrete topology.  In particular, if $V$ is any vector space over any field, then every linear cellular automaton $V^G\to V^G$ has closed image.
\end{proposition}

\begin{proof}
Let $M\subset G$ be a finite memory set for $\tau$, and let $H=\langle M\rangle$.  Since $G$ is locally finite, $H$ is finite.  Let $(C_\lambda)_{\lambda\in\Lambda}$ be the family of left cosets of $H$ in $G$, so that
\[
        G=\bigsqcup_{\lambda\in\Lambda} C_\lambda,
        \qquad
        C_\lambda=g_\lambda H
        \quad(\lambda\in\Lambda).
\]
For $g\in C_\lambda$ one has $gM\subset C_\lambda$, because $M\subset H$.  Hence $\tau(x)(g)$ depends only on the restriction of $x$ to the same coset $C_\lambda$.  Thus $\tau$ decomposes as a product of maps
\[
        \tau_\lambda\colon A^{C_\lambda}\longrightarrow A^{C_\lambda},
        \qquad \lambda\in\Lambda.
\]
Since $H$ is finite, every coset $C_\lambda$ is finite, and hence $A^{C_\lambda}$ is discrete for the prodiscrete topology.  Therefore $\tau_\lambda(A^{C_\lambda})$ is closed in $A^{C_\lambda}$.  Under the natural product identification
\[
        A^G=\prod_{\lambda\in\Lambda}A^{C_\lambda},
\]
the image of $\tau$ is
\[
        \tau(A^G)=\prod_{\lambda\in\Lambda}\tau_\lambda(A^{C_\lambda}),
\]
which is closed as a product of closed subsets.  This proves the assertion.
\end{proof}

\subsection{Infinite rays in locally finite Cayley graphs}

We shall use the following classical form of K\"onig's infinity lemma; see, for example, \cite[Section 8.1]{Diestel}.

\begin{lemma}[K\"onig's infinity lemma]\label{lem:konig}
Every infinite rooted tree in which each vertex has finitely many children contains an infinite ray starting at the root.
\end{lemma}

\begin{lemma}\label{lem:ray}
Let $G$ be a non-locally finite group.  Then there exist a finite subset $S\subset G$ and a sequence $(s_n)_{n\geq 1}$ in $S$ such that the elements
\[
        p_0=1_G,
        \qquad
        p_n=s_1s_2\cdots s_n \quad(n\geq 1)
\]
are pairwise distinct.
\end{lemma}

\begin{proof}
Choose a finitely generated infinite subgroup $H\leq G$, and let $S\subset H$ be a finite symmetric generating set with $1_G\notin S$.  Consider the right Cayley graph $\Cay(H,S)$, with vertices $H$ and directed edges $h\to hs$ for $h\in H$ and $s\in S$.  It is connected, infinite, and locally finite.

Fix an ordering of the finite set $S$.  Define a rooted geodesic spanning tree $T$ of $\Cay(H,S)$ as follows.  For each vertex $h\neq 1_G$, let $s(h)$ be the first element $s\in S$ such that $d_S(1_G,hs^{-1})=d_S(1_G,h)-1$, and declare $hs(h)^{-1}$ to be the predecessor of $h$.  Such an element exists by the definition of word distance.  The resulting rooted graph is a tree: each non-root vertex has exactly one predecessor, and along every predecessor chain the word distance from $1_G$ strictly decreases, so no cycle can occur.  It is locally finite and infinite.  By K\"onig's infinity lemma, $T$ contains an infinite ray starting at $1_G$.  Equivalently, since the children are labelled by the fixed finite ordered set $S$, the ray may be chosen recursively by taking the first child whose rooted subtree is infinite; no additional choice is involved beyond this finite ordering.  If $s_n$ is the label of the $n$-th edge of this ray, then
\[
        p_n=s_1s_2\cdots s_n
\]
are the successive vertices of the ray, and hence are pairwise distinct.
\end{proof}

\section{The countable direct-sum construction}\label{sec:construction}

Throughout this section, fix a non-locally finite group $G$ and choose $S$ and $(s_n)_{n\geq 1}$ as in \cref{lem:ray}.  Let
\[
        E=\bigoplus_{n\geq 0}K e_n.
\]
Every vector $v\in E$ has a unique finite expansion
\[
        v=\sum_{n\geq 0}v_n e_n,
        \qquad v_n\in K,
\]
where only finitely many $v_n$ are nonzero.

Define the right shift operator
\[
        R\colon E\longrightarrow E,
        \qquad R(e_n)=e_{n+1}\quad(n\geq 0).
\]
Let $P_0\colon E\to E$ be the projection onto $Ke_0$:
\[
        P_0(e_0)=e_0,
        \qquad
        P_0(e_n)=0\quad(n\geq 1).
\]
For each $s\in S$, define a $K$-linear map $P_s\colon E\to E$ by
\begin{equation}\label{eq:Ps}
        P_s(e_n)=
        \begin{cases}
        e_n,& n\geq 1\text{ and }s_n=s,\\
        0,& \text{otherwise.}
        \end{cases}
\end{equation}
This is well defined: although the set of indices $n$ with $s_n=s$ may be infinite, each vector of $E$ has finite support.

Let
\[
        M=\{1_G\}\cup S.
\]
Define
\begin{equation}\label{eq:tauE}
        \tau_E(x)(g)
        =P_0(x(g))-R(x(g))+\sum_{s\in S}P_s(x(gs)),
        \qquad x\in E^G,\ g\in G.
\end{equation}

\begin{lemma}\label{lem:tau-ca}
The map $\tau_E\colon E^G\to E^G$ is a linear cellular automaton with memory set $M$.
\end{lemma}

\begin{proof}
Define a $K$-linear map $\mu\colon E^M\to E$ by
\[
        \mu((v_m)_{m\in M})
        =P_0(v_{1_G})-R(v_{1_G})+
        \sum_{s\in S}P_s(v_s).
\]
Then, for $x\in E^G$ and $g\in G$,
\[
        \mu\bigl((g^{-1}x)|_M\bigr)
        =P_0(x(g))-R(x(g))+
        \sum_{s\in S}P_s(x(gs)),
\]
because $(g^{-1}x)(m)=x(gm)$.  Hence \eqref{eq:tauE} is precisely the cellular automaton with memory set $M$ and local defining map $\mu$.
\end{proof}

For $x\in E^G$, write
\[
        x(g)=\sum_{n\geq 0}x_n(g)e_n,
        \qquad x_n(g)\in K.
\]

\begin{lemma}[Coordinate form]\label{lem:coordinate}
For every $x\in E^G$ and $g\in G$, one has
\begin{equation}\label{eq:coord0}
        (\tau_E(x))_0(g)=x_0(g),
\end{equation}
and, for every $n\geq 1$,
\begin{equation}\label{eq:coordn}
        (\tau_E(x))_n(g)=x_n(gs_n)-x_{n-1}(g).
\end{equation}
\end{lemma}

\begin{proof}
The term $P_0(x(g))$ contributes only to the zeroth coordinate, and the term $-R(x(g))$ contributes $-x_{n-1}(g)$ to the $n$-th coordinate for $n\geq 1$.  In the sum $\sum_{s\in S}P_s(x(gs))$, the $e_n$-coordinate with $n\geq 1$ receives exactly one contribution, namely from the summand with $s=s_n$, and that contribution is $x_n(gs_n)$.  This proves the formulas.
\end{proof}

Let $y\in E^G$ be the constant configuration
\begin{equation}\label{eq:y}
        y(g)=e_0
        \qquad(g\in G).
\end{equation}
We shall prove that $y$ belongs to the closure of $\tau_E(E^G)$ but not to $\tau_E(E^G)$.

\section{The non-closed image}\label{sec:nonclosed}

\subsection{No global preimage}

\begin{proposition}\label{prop:not-image}
The configuration $y$ does not belong to $\tau_E(E^G)$.
\end{proposition}

\begin{proof}
Suppose that $\tau_E(x)=y$ for some $x\in E^G$.  From the zeroth-coordinate identity \eqref{eq:coord0}, we get
\[
        x_0(g)=1
        \qquad\text{for all }g\in G.
\]
For $n\geq 1$, the $n$-th coordinate of $y$ is zero, so \eqref{eq:coordn} gives
\begin{equation}\label{eq:recursion-global}
        x_n(gs_n)=x_{n-1}(g)
        \qquad\text{for all }g\in G.
\end{equation}
Right multiplication by $s_n$ is a bijection of $G$.  Therefore, if $x_{n-1}$ is the constant function $1$, then $x_n$ is also the constant function $1$.  By induction,
\[
        x_n(g)=1
        \qquad\text{for all }n\geq 0\text{ and all }g\in G.
\]
Thus, for any fixed $g\in G$, the vector $x(g)$ would have infinitely many nonzero coordinates:
\[
        x(g)=e_0+e_1+e_2+\cdots.
\]
This vector does not belong to the algebraic direct sum $E=\bigoplus_{n\geq 0}Ke_n$.  This contradiction proves that $y\notin\tau_E(E^G)$.
\end{proof}

\subsection{Local liftability on every finite set}

\begin{proposition}\label{prop:closure}
The configuration $y$ belongs to the closure of $\tau_E(E^G)$ in the prodiscrete topology.
\end{proposition}

\begin{proof}
Let $\Omega\subset G$ be finite.  We shall construct $x^\Omega\in E^G$ such that
\begin{equation}\label{eq:local-target}
        \tau_E(x^\Omega)(g)=e_0
        \qquad\text{for all }g\in\Omega.
\end{equation}
This is exactly the condition that the basic neighborhood $\mathcal N(y,\Omega)$ meets $\tau_E(E^G)$.

Define finite subsets $A_n\subset G$ recursively by
\begin{equation}\label{eq:An-recursion}
        A_0=\Omega,
        \qquad
        A_n=(A_{n-1}\cap\Omega)s_n\quad(n\geq 1),
\end{equation}
where $Bs=\{bs:b\in B\}$.

We first claim that
\begin{equation}\label{eq:An-explicit}
        A_n=
        \{gp_n:g\in\Omega,
        \ gp_j\in\Omega\text{ for all }0\leq j<n\}.
\end{equation}
For $n=0$ this says $A_0=\Omega$.  Suppose it holds for $n-1$.  An element of $A_n$ has the form $a s_n$, where $a\in A_{n-1}\cap\Omega$.  By induction, $a=gp_{n-1}$ for some $g\in\Omega$ with $gp_j\in\Omega$ for $0\leq j<n-1$, and the additional condition $a\in\Omega$ says $gp_{n-1}\in\Omega$.  Hence $a s_n=gp_n$, and the conditions hold for all $0\leq j<n$.  This proves the claim.

Since the elements $p_0,p_1,p_2,\ldots$ are pairwise distinct, the elements
\[
        gp_0,gp_1,\ldots,gp_{n-1}
\]
are pairwise distinct for each fixed $g\in G$.  Hence, if $A_n\neq\varnothing$, then \eqref{eq:An-explicit} gives $n$ distinct elements of the finite set $\Omega$.  Therefore $n\leq |\Omega|$.  Consequently,
\begin{equation}\label{eq:An-empty}
        A_n=\varnothing
        \qquad(n>|\Omega|).
\end{equation}

Define $x^\Omega\in E^G$ by
\begin{equation}\label{eq:xOmega}
        x^\Omega(h)=\sum_{n\geq 0}\chi_{A_n}(h)e_n,
        \qquad h\in G,
\end{equation}
where $\chi_{A_n}$ denotes the characteristic function of $A_n$, with values in $\{0,1\}\subset K$.  The sum is finite by \eqref{eq:An-empty}, so $x^\Omega(h)\in E$ for every $h\in G$.

Now fix $g\in\Omega$.  Since $A_0=\Omega$, we have
\[
        x^\Omega_0(g)=1.
\]
For $n\geq 1$, using \eqref{eq:An-recursion} and $g\in\Omega$, we obtain
\[
        x^\Omega_n(gs_n)
        =\chi_{A_n}(gs_n)
        =\chi_{A_{n-1}\cap\Omega}(g)
        =\chi_{A_{n-1}}(g)
        =x^\Omega_{n-1}(g).
\]
By \cref{lem:coordinate}, this gives
\[
        (\tau_E(x^\Omega))_0(g)=1,
        \qquad
        (\tau_E(x^\Omega))_n(g)=0\quad(n\geq 1).
\]
Thus \eqref{eq:local-target} holds.  Since $\Omega$ was arbitrary, $y\in\overline{\tau_E(E^G)}$.
\end{proof}

Combining the two propositions gives the countable-direct-sum case.

\begin{theorem}\label{thm:direct-sum}
Let $G$ be a non-locally finite group and let $K$ be a field.  Put
\[
        E=\bigoplus_{n\geq 0}K e_n.
\]
Then there exists a linear cellular automaton $\tau_E\colon E^G\to E^G$ such that
\[
        \overline{\tau_E(E^G)}\neq \tau_E(E^G).
\]
More precisely, the constant configuration $y(g)=e_0$ satisfies
\[
        y\in\overline{\tau_E(E^G)}\setminus\tau_E(E^G).
\]
\end{theorem}

\begin{proof}
Use the automaton \eqref{eq:tauE}.  It is a linear cellular automaton by \cref{lem:tau-ca}.  The non-image statement is \cref{prop:not-image}, and the closure statement is \cref{prop:closure}.
\end{proof}

\section{Two extensions of the alphabet}\label{sec:alphabet-extensions}

The countable direct-sum construction admits two complementary extension procedures.  The first preserves linearity and passes to arbitrary infinite-dimensional vector spaces.  The second forgets linear structure and passes along a retract of alphabets.

\subsection{Direct summands of vector-space alphabets}

\begin{proposition}[Direct-summand transfer]\label{prop:direct-summand-transfer}
Let $E$ be a $K$-vector space, let $V=E\oplus W$, and suppose that a linear cellular automaton $\sigma\colon E^G\to E^G$ has non-closed image.  Then there exists a linear cellular automaton $\widehat\sigma\colon V^G\to V^G$ with non-closed image.  Its memory set may be chosen to be the union of a memory set for $\sigma$ with $\{1_G\}$.
\end{proposition}

\begin{proof}
Under the natural linear homeomorphism $V^G\cong E^G\times W^G$, define
\[
        \widehat\sigma=\sigma\times\id_{W^G}.
\]
If $M$ is a memory set for $\sigma$, then $M\cup\{1_G\}$ is a memory set for $\widehat\sigma$, and the corresponding local map is linear.  Moreover,
\[
        \widehat\sigma(V^G)=\sigma(E^G)\times W^G.
\]
Choose $y\in\overline{\sigma(E^G)}\setminus\sigma(E^G)$.  Then
\[
        (y,0)\in
        \overline{\sigma(E^G)\times W^G}
        \setminus
        \bigl(\sigma(E^G)\times W^G\bigr),
\]
so the image of $\widehat\sigma$ is not closed.
\end{proof}

\begin{proof}[Proof of \cref{thm:main-intro}\textup{(b)}]
\leavevmode\par
Let $K$ be a field and let $V$ be an infinite-dimensional $K$-vector space.  Choose a countably infinite linearly independent family $(e_n)_{n\geq 0}$ in $V$ and put
\[
        E=\bigoplus_{n\geq 0}K e_n.
\]
Choose a vector-space complement $W$ of $E$ in $V$.  By \cref{thm:direct-sum}, there is a linear cellular automaton $\tau_E\colon E^G\to E^G$ with non-closed image.  Apply \cref{prop:direct-summand-transfer} to $V=E\oplus W$.
\end{proof}

\subsection{Retracts of set alphabets}

\begin{proposition}[Alphabet-retract transfer]\label{prop:alphabet-retract}
Let $A$ and $B$ be sets, and suppose that there are maps
\[
        \iota\colon B\longrightarrow A,
        \qquad
        r\colon A\longrightarrow B,
        \qquad
        r\circ\iota=\id_B.
\]
If there exists a cellular automaton $\sigma\colon B^G\to B^G$ with non-closed image, then there exists a cellular automaton $\tau\colon A^G\to A^G$ with non-closed image.  The automaton $\tau$ may be chosen with the same memory set as $\sigma$.
\end{proposition}

\begin{proof}
Let $M$ be a memory set for $\sigma$, with local defining map $\mu\colon B^M\to B$.  Define
\[
        \nu=\iota\circ\mu\circ r^M\colon A^M\longrightarrow A,
\]
and let $\tau\colon A^G\to A^G$ be the cellular automaton with memory set $M$ and local defining map $\nu$.  Coordinatewise, one has
\begin{equation}\label{eq:alphabet-transfer-identity}
        \tau=\iota^G\circ\sigma\circ r^G.
\end{equation}
Since $r^G\circ\iota^G=\id_{B^G}$, the map $r^G$ is surjective and therefore
\begin{equation}\label{eq:alphabet-transfer-image}
        \tau(A^G)=\iota^G\bigl(\sigma(B^G)\bigr).
\end{equation}

Choose
\[
        y\in\overline{\sigma(B^G)}\setminus\sigma(B^G).
\]
For every finite $\Omega\subset G$, there exists $z_\Omega\in B^G$ such that
\[
        \sigma(z_\Omega)|_\Omega=y|_\Omega.
\]
Using $x_\Omega=\iota^G(z_\Omega)$ in \eqref{eq:alphabet-transfer-identity}, we obtain
\[
        \tau(x_\Omega)|_\Omega=\iota^G(y)|_\Omega.
\]
Thus $\iota^G(y)\in\overline{\tau(A^G)}$.  If $\iota^G(y)$ belonged to $\tau(A^G)$, then \eqref{eq:alphabet-transfer-image} and the injectivity of $\iota^G$ would give $y\in\sigma(B^G)$, a contradiction.  Hence
\[
        \iota^G(y)\in\overline{\tau(A^G)}\setminus\tau(A^G),
\]
and the image of $\tau$ is not closed.
\end{proof}

\begin{proof}[Proof of \cref{thm:main-intro}\textup{(a)}]
\leavevmode\par
Take $K=\mathbb F_2$ in \cref{thm:direct-sum} and write
\[
        B=\bigoplus_{n\geq 0}\mathbb F_2 e_n
\]
for the underlying set of the resulting vector-space alphabet.  The set $B$ is countably infinite, and the linear cellular automaton supplied by \cref{thm:direct-sum} is, in particular, a cellular automaton $B^G\to B^G$ with non-closed image.

Let $A$ be any infinite set.  Since we work in ZFC, $A$ contains a countably infinite subset; equivalently, there is an injection $\iota\colon B\to A$ because $B$ is countably infinite.  Fix $b_0\in B$ and define a retraction $r\colon A\to B$ by
\[
        r(\iota(b))=b\quad(b\in B),
        \qquad
        r(a)=b_0\quad(a\notin\iota(B)).
\]
Then $r\circ\iota=\id_B$, so \cref{prop:alphabet-retract} gives a cellular automaton $A^G\to A^G$ with non-closed image.
\end{proof}

\subsection{The open problems and the characterization}

\begin{proof}[Proof of \cref{cor:open-problems}]
A periodic group which is not locally finite is, in particular, not locally finite.  Part \textup{(a)} of \cref{thm:main-intro} answers OP-6, and part \textup{(b)} answers OP-7.
\end{proof}

\begin{proof}[Proof of \cref{thm:characterization-intro}]
By \cref{prop:locally-finite-closed}, condition \textup{(i)} implies \textup{(ii)}.  Clearly, \textup{(ii)} implies \textup{(iii)}, and \textup{(iii)} implies \textup{(iv)}.  If $G$ is not locally finite, then \cref{thm:main-intro}\textup{(a)} produces a cellular automaton with non-closed image over every infinite alphabet.  Hence \textup{(iv)} implies \textup{(i)}.

Similarly, \cref{prop:locally-finite-closed} shows that \textup{(i)} implies \textup{(v)}, and \textup{(v)} plainly implies \textup{(vi)}.  If $G$ is not locally finite, then \cref{thm:main-intro}\textup{(b)} produces a linear cellular automaton with non-closed image over every infinite-dimensional vector-space alphabet.  Hence \textup{(vi)} implies \textup{(i)}.  This completes both cycles of equivalences.
\end{proof}

\begin{remark}[The ordinary and linear thresholds differ]\label{rem:thresholds}
Let $K$ be an infinite field and let $V$ be a nonzero finite-dimensional $K$-vector space.  Then $V$ is infinite as a set.  If $G$ is not locally finite, \cref{thm:main-intro}\textup{(a)} gives a cellular automaton $V^G\to V^G$ with non-closed image, whereas every \emph{linear} cellular automaton on the same configuration space has closed image \cite[Theorem 8.8.1]{CAG2}.  Thus the arbitrary-alphabet statement is genuinely stronger than the linear statement on a fixed alphabet; the bridge from OP-7 to OP-6 is the retract construction, not a linear structure on the prescribed alphabet.
\end{remark}

\section{The Grigorchuk test case}\label{sec:grigorchuk}

We spell out the standard periodic example to which both OP-6 and OP-7 are naturally attached.

\begin{definition}[The first Grigorchuk group]\label{def:grigorchuk}
Let $\Sigma=\{0,1\}$, and let $\Sigma^*$ be the rooted binary tree of all finite words over $\Sigma$, including the empty word $\eps$.  The first Grigorchuk group $\Grig$ is the subgroup of $\Aut(\Sigma^*)$ generated by four automorphisms $a,b,c,d$ defined recursively as follows.  The automorphism $a$ fixes the root and swaps the two first-level subtrees:
\[
        a(\eps)=\eps,
        \qquad
        a(0w)=1w,
        \qquad
        a(1w)=0w.
\]
The remaining generators fix the root and satisfy, for every $w\in\Sigma^*$,
\[
\begin{aligned}
        b(0w)&=0a(w), & b(1w)&=1c(w),\\
        c(0w)&=0a(w), & c(1w)&=1d(w),\\
        d(0w)&=0w,    & d(1w)&=1b(w).
\end{aligned}
\]
Thus
\[
        \Grig=\langle a,b,c,d\rangle\leq \Aut(\Sigma^*).
\]
Equivalently, in wreath-recursion notation,
\[
        a=(1,1)\sigma,
        \qquad
        b=(a,c),
        \qquad
        c=(a,d),
        \qquad
        d=(1,b),
\]
where $\sigma$ is the transposition of the two first-level vertices.
\end{definition}

It is classical that $\Grig$ is finitely generated, infinite, and periodic; it is one of the original examples arising from Grigorchuk's work on the Burnside problem and groups of intermediate growth \cite{Grigorchuk1984,CAG2}.  Since a finitely generated locally finite group is finite, $\Grig$ is not locally finite.

\begin{corollary}[The first Grigorchuk group]\label{cor:grigorchuk}
Let $\Grig$ be the first Grigorchuk group.
\begin{enumerate}[label=\textup{(\alph*)}]
\item For every infinite set $A$, there exists a cellular automaton $\tau\colon A^{\Grig}\to A^{\Grig}$ whose image is not closed in the prodiscrete topology.
\item For every field $K$ and every infinite-dimensional $K$-vector space $V$, there exists a linear cellular automaton $\tau\colon V^{\Grig}\to V^{\Grig}$ whose image is not closed in the prodiscrete topology.
\end{enumerate}
\end{corollary}

\begin{proof}
By the preceding discussion, $\Grig$ is not locally finite.  Apply \cref{thm:main-intro}.
\end{proof}

\begin{remark}[What replaces an infinite-order direction]\label{rem:replacement}
In the non-periodic case treated by Ceccherini-Silberstein and Coornaert, an element of infinite order gives a one-sided direction along which a shift-type obstruction may be placed.  A periodic group has no such element.  The present construction shows that the infinite-order direction is not necessary: an infinite simple ray in finitely many generating directions is enough.  The labels of this ray may vary, but the direct-sum alphabet stores the $n$-th label in its $n$-th coordinate through the projections $P_s$ in \eqref{eq:Ps}.  This is the mechanism that removes the periodicity obstruction.
\end{remark}

\begin{remark}[Finite alphabets and finite-dimensional linear alphabets]
For a finite set $A$, compactness of $A^G$ in the usual ZFC framework implies that the image of every cellular automaton $A^G\to A^G$ is closed, for every group $G$.  For a finite-dimensional vector space $V$, every linear cellular automaton $V^G\to V^G$ has closed image, even when the underlying set of $V$ is infinite \cite[Theorem 8.8.1]{CAG2}.  The results above show that, once arbitrary infinite alphabets or infinite-dimensional linear alphabets are allowed, local finiteness is exactly the group-theoretic condition forcing the closed image property.
\end{remark}

\end{document}